\documentclass[11pt]{article}

\usepackage{amsmath,amssymb}
\usepackage{latexsym}
\usepackage{graphicx}
\usepackage{bm}

\newtheorem{thm}{Theorem}%[section]
\newtheorem{cor}[thm]{Corollary}

\newtheorem{lem}[thm]{Lemma}

\newtheorem{rem}{Remark}

\newenvironment{proof}{\begin{trivlist}
                       \item[]{\bf Proof.}
                       \hspace{0cm}}{\hfill $\Box$
                       \end{trivlist}}

\begin{document}

\title{A discrepancy principle for equations with monotone continuous operators}

\author{N. S. Hoang$\dag$\footnotemark[1]\, and\, A. G. Ramm$
\dag$\footnotemark[3]
\\
\\
$\dag$Mathematics Department, Kansas State University,\\
Manhattan, KS 66506-2602, USA
}

\renewcommand{\thefootnote}{\fnsymbol{footnote}}
\footnotetext[3]{Corresponding author. Email: ramm@math.ksu.edu}
\footnotetext[1]{Email: nguyenhs@math.ksu.edu}

\date{}
\maketitle

\begin{abstract} A discrepancy principle for solving nonlinear equations
with monotone operators given noisy data is formulated.  The existence and
uniqueness of the corresponding regularization parameter $a(\delta)$ is
proved. Convergence of the solution obtained by the discrepancy principle
is justified.  The results are obtained under natural assumptions on the
nonlinear operator. \end{abstract}



\section{Introduction}

Consider the equation:
\begin{equation} 
\label{eq1} 
F(u)=f,
\end{equation} 
where $F$ is a monotone operator in a real Hilbert
space $H$. Monotonicity is understood in the following sense:
\begin{equation}
\label{eq2}
\langle F(u)-F(v),u-v\rangle \ge 0,\quad \forall u,v\in H.
\end{equation}
Here $\langle \cdot, \cdot\rangle$ denotes the inner product in $H$. 
Assume that $F$ is continuous.
% and satisfies the following condition
%\begin{equation}
%\label{eq3}
%\|F(u) - F(0)\| \le M_1(R)\|u\|,\qquad \forall u \in B(0,R).
%\end{equation}

Equations with monotone operators are important in many applications and
were studied extensively, see, for example, \cite{D}--\cite{Pascali},
\cite{R499}, \cite{Skrypnik}, \cite{Vainberg}, and references therein.  
There are many technical and physical problems leading to equations with
such operators in the cases when dissipation of energy occurs. For
example, in \cite{[9]} and \cite{[8]}, Chapter 3, pp.156-189, a wide class
of nonlinear dissipative systems is studied, and the basic equations of
such systems can be reduced to equation (1) with monotone operators. Many
examples of equations with monotone operators can be found in \cite{Lions}
and in references mentioned above.  In \cite{[19]} and \cite{[20]} it is
proved that any solvable linear operator equation with a closed densely
defined operator in a Hilbert space $H$ can be reduced to an equation with
a monotone operator and solved by a convergent iterative process.

In this paper, apparently for the first time, a discrepancy principle for
solving equation \eqref{eq3} with noisy data (see Section 2) is proved
under natural assumptions.  No smallness assumptions on the nonlinearity,
no global restrictions on its growth, or other special properties of the
nonlinearity, except the monotonicity and continuity, are imposed. No
source-type assumptions are used. Our result is widely applicable. It is
well known that without extra assumptions, usually source-type assumption
concerning the right-hand side, or some equivalent assumption concerning
the smoothness of the solution, one cannot get a rate of convergence even
for linear ill-posed equations (see, for example, \cite{R499}).
On the other hand, such assumptions
are usually not algorithmically verifiable and often they do not hold. By
this reason we do not make such assumptions
 and do not give estimates of the rate of convergence. 

In \cite{Tautenhahn} a stationary equation $F(u) = f$ with a nonlinear
monotone operator $F$ was studied. The assumptions A1-A3 on p.197 in
\cite{Tautenhahn} are more restrictive than ours, and the Rule R2 on
p.199, formula (4.1) in \cite{Tautenhahn}, for the choice of the
regularization parameter is more difficult to use computationally: one has
to solve nonlinear equation (4.1) in \cite{Tautenhahn} for the
regularization parameter. Moreover, to use this equation one has to invert
an ill-conditioned linear operator $A+aI$ for small values of $a$.
Assumption A1 in \cite{Tautenhahn} is not verifiable, because the solution
$x^\dag$ is not known. Assumption A3 in \cite{Tautenhahn} requires $F$ to
be constant in a ball $B_r(x^\dag)$ if $F'(x^\dag)$ = 0. Our discrepancy
principle does not require these assumptions, and, in contrast to equation
(4.1) in \cite{Tautenhahn}, it does not require inversion of
ill-conditioned linear operators.

The novel results in our paper include Theorem~\ref{thm1} in
Section~\ref{sec3} and Theorem~\ref{thm2} in Section~\ref{sec4}. In
Theorem~\ref{thm1} a new discrepancy principle is proposed and justified
assuming only the monotonicity and continuity of $F$. Implementing the
discrepancy principle in Theorem~\ref{thm1} requires solving equation
\eqref{eq3} and then solving nonlinear equation \eqref{eq131} for the
regularization parameter $a(\delta)$. Theorem~\ref{thm2} allows one to
solve equations \eqref{eq3} and \eqref{eq131} approximately. Thus, when
$\delta$ is not too small one can save a large amount of computations in
solving equations \eqref{eq3} and \eqref{eq131} by applying
Theorem~\ref{thm2} and using our new stopping rule.  Our results allow
one to solve numerically stably equation \eqref{eq1} if $F$ is locally
Lipschitz and monotone. Based on Theorem~\ref{thm2}, an algorithm for
stable solution of equation \eqref{eq1} is formulated for locally
Lipschitz monotone operators.

\section{Auxiliary results}

Let us consider the following equation
\begin{equation}
\label{eq3}
F(V_{\delta,a})+aV_{\delta,a}-f_\delta = 0,\qquad a>0,
\end{equation}
where $a=const$. It is known (see, e.g., \cite[p.111]{R499}) 
that equation \eqref{eq3} with monotone continuous operator $F$ has 
a unique solution for any $f_\delta\in H$. 

Throughout the paper we assume that $F$ is a monotone continuous operator
and the inner product in $H$ is denoted $\langle u,v\rangle$. 
Below the word decreasing means strictly decreasing and increasing means strictly increasing. 

Recall the following result from \cite[p.112]{R499}:
\begin{lem}
\label{lem1}
Assume that equation \eqref{eq1} is solvable, $y$ is its minimal-norm solution, assumption
\eqref{eq2} holds, and $F$ is continuous. Then
\begin{equation}
\lim_{a\to 0} \|V_{a}-y\| = 0,
\end{equation}
where $V_{a}$ solves equation \eqref{eq3} with $\delta=0$.
\end{lem}

\begin{lem}
\label{lem2}
Assume $\|F(0)-f_\delta\|>0$.
Let $a>0$, and $F$ be monotone.
Denote 
$$
\psi(a) :=\|V_{\delta,a} \|,\qquad \phi(a):=a\psi(a)=\|F(V_{\delta,a}) - f_\delta\|,
$$ 
where $V_{\delta,a}$ solves \eqref{eq3}. 
Then
$\psi(a)$ is decreasing, and $\phi(a)$ is increasing.
\end{lem}
 
\begin{proof}
Since $\|F(0)-f_\delta \|>0$, one has $\psi(a)\not=0,\, \forall a\ge 0$. 
Indeed, if $\psi(a)\big{|}_{a=\tau}=0$, then $V_{\delta,a}=0$, and equation \eqref{eq3} implies
$\|F(0)-f_\delta\|=0$, which is a contradiction. 
Note that $\phi(a)=a\|V_{\delta,a}\|$. One has
\begin{equation}
\label{1eq3}
\begin{split}
0&\le \langle F(V_{\delta,a})-F(V_{\delta,b}),V_{\delta,a}-V_{\delta,b}\rangle\\
&= \langle -aV_{\delta,a}+bV_{\delta,b},V_{\delta,a}-V_{\delta,b}\rangle\\
&= (a+b)\langle V_{\delta,a},V_{\delta,b} \rangle -a\|V_{\delta,a}\|^2 - b\|V_{\delta,b}\|^2.
\end{split}
\end{equation}
Thus,
\begin{equation}
\label{2eq6}
\begin{split}
0&\le (a+b)\langle V_{\delta,a},V_{\delta,b} \rangle -a\|V_{\delta,a}\|^2 - b\|V_{\delta,b}\|^2\\
& \le  (a+b)\|V_{\delta,a}\|\|V_{\delta,b} \| - a\|V_{\delta,a}\|^2 - b\|V_{\delta,b}\|^2\\
& = (a \|V_{\delta,a}\| - b \|V_{\delta,b}\|)(\|V_{\delta,b}\|-\|V_{\delta,a}\|)\\
& = (\phi(a)-\phi(b))(\psi(a) - \psi(b)).
\end{split}
\end{equation}

If $\psi(b) > \psi(a)$ then \eqref{2eq6} implies $\phi(a)\ge \phi(b)$, so
$$
a\psi(a)\ge b\psi(b)> b\psi(a).
$$
Therefore, if $\psi(b)> \psi(a)$ then $b< a$. 

Similarly, if $\psi(b)< \psi(a)$ then $\phi(a)\le \phi(b)$. This implies $b> a$.

Suppose $\psi(a)=\psi(b)$, i.e., $\|V_{\delta,a}\|=\|V_{\delta,b}\|$. 
From \eqref{1eq3} one has
$$
\|V_{\delta,a}\|^2\le \langle V_{\delta,a},V_{\delta,b} \rangle \le \|V_{\delta,a}\|\|V_{\delta,b}\| = \|V_{\delta,a}\|^2.
$$
This implies $V_{\delta,a}=V_{\delta,b}$, and then equation \eqref{eq3} implies $a=b$. 

Therefore $\phi$ is increasing
and $\psi$ is decreasing.
\end{proof}

\begin{lem}
\label{lem0}
If $F$ is monotone and continuous, then
$\|V_{\delta,a}\|=O(\frac{1}{a})$ as $a\to\infty$, and
\begin{equation}
\label{4eq2}
\lim_{a\to\infty}\|F(V_{\delta,a})-f_\delta\|=\|F(0)-f_\delta\|.
\end{equation}
\end{lem}

\begin{proof}
%First, we claim that $aV_{\delta,a}$ is bounded for all $a>0$. Indeed, 
Rewrite \eqref{eq3} as
$$
F(V_{\delta,a}) - F(0) + aV_{\delta,a} + F(0)-f_\delta = 0.
$$
Multiply this equation
by $V_{\delta,a}$, use 
%inequality $\langle F(V_{\delta,a})-F(0),V_{\delta,a}-0\rangle \ge 0$ from %\eqref{eq2} 
the monotonicity of $F$ and get:
$$
a\|V_{\delta,a}\|^2\le\langle a V_{\delta,a} + F(V_{\delta,a})-F(0), V_{\delta,a}\rangle = 
\langle f_\delta-F(0), V_{\delta,a}\rangle \le \|f_\delta-F(0)\|\|V_{\delta,a}\|.
$$
Therefore,
$\|V_{\delta,a}\|=O(\frac{1}{a})$. This and the continuity of $F$ imply \eqref{4eq2}.
\end{proof}

\begin{rem}
\label{rem1}
{\rm
If $\|F(0)-f_\delta\|>C\delta^\gamma$,\,$0<\gamma\le 1$ then relation \eqref{4eq2} implies 
\begin{equation}
\label{eqrem1}
\|F(V_{\delta,a})-f_\delta\|\ge C\delta^\gamma,\qquad 0<\gamma\le 1,
\end{equation}
for sufficiently large $a>0$. 
}
\end{rem}

\begin{lem}
\label{lem4}
Let $C>0$ and $\gamma\in (0,1]$ be constants such that $C\delta^\gamma>\delta$.
Suppose that $\|F(0)-f_\delta\|> C\delta^\gamma$. 
Then, there exists a unique $a(\delta)>0$ such 
that $\|F(V_{\delta,a(\delta)})-f_\delta\|=C\delta^\gamma$.
\end{lem}

\begin{proof}
We have $F(y)=f$, and
\begin{align*}
0 = &\langle F(V_{\delta,a})+aV_{\delta,a}-f_\delta, F(V_{\delta,a})-f_\delta \rangle\\
 = &\|F(V_{\delta,a})-f_\delta\|^2+a\langle V_{\delta,a}-y, F(V_{\delta,a})-f_\delta \rangle + a\langle y, F(V_{\delta,a})-f_\delta \rangle\\
 = &\|F(V_{\delta,a})-f_\delta\|^2+a\langle V_{\delta,a}-y, F(V_{\delta,a})-F(y) \rangle + a\langle V_{\delta,a}-y, f-f_\delta \rangle \\
 &+ a\langle y, F(V_{\delta,a})-f_\delta \rangle\\
\ge & \|F(V_{\delta,a})-f_\delta\|^2 + a\langle V_{\delta,a}-y, f-f_\delta \rangle + a\langle y, F(V_{\delta,a})-f_\delta \rangle.
\end{align*}
Here the 
%inequality $\langle V_{\delta,a}-y, F(V_{\delta,a})-F(y) \rangle\ge0$ 
monotonicity of $F$ was used. 
Therefore
\begin{equation}
\label{1eq1}
\begin{split}
\|F(V_{\delta,a})-f_\delta\|^2 &\le -a\langle V_{\delta,a}-y, f-f_\delta \rangle - a\langle y, F(V_{\delta,a})-f_\delta \rangle\\
&\le a\|V_{\delta,a}-y\| \|f-f_\delta\| + a\|y\| \|F(V_{\delta,a})-f_\delta\|\\
&\le  a\delta \|V_{\delta,a}-y\|  + a\|y\| \|F(V_{\delta,a})-f_\delta\|.
\end{split}
\end{equation}
Also,
\begin{align*}
0&= \langle F(V_{\delta,a})-F(y) + aV_{\delta,a} +f -f_\delta, V_{\delta,a}-y\rangle\\
&=\langle F(V_{\delta,a})-F(y),V_{\delta,a}-y\rangle + a\| V_{\delta,a}-y\| ^2 + a\langle y, V_{\delta,a}-y\rangle + \langle f-f_\delta, V_{\delta,a}-y\rangle\\
&\ge  a\| V_{\delta,a}-y\| ^2 + a\langle y, V_{\delta,a}-y\rangle + 
\langle f-f_\delta, V_{\delta,a}-y\rangle,
\end{align*}
where the 
%inequality $\langle V_{\delta,a}-y, F(V_{\delta,a})-F(y) \rangle\ge0$ 
monotonicity of $F$ was used again. Therefore,
$$
a\|V_{\delta,a}-y\|^2 \le a\|y\|\|V_{\delta,a}-y\|+\delta\|V_{\delta,a}-y\|.
$$
This implies
\begin{equation}
\label{1eq2}
a\|V_{\delta,a}-y\|\le a\|y\|+\delta.
\end{equation}
From \eqref{1eq1}, \eqref{1eq2}, and an elementary inequality $ab\le \epsilon a^2+\frac{b^2}{4\epsilon},\,\forall\epsilon>0$, one gets:
\begin{equation}
\label{3eq4}
\begin{split}
\|F(V_{\delta,a})-f_\delta\|^2&\le \delta^2 + a\|y\|\delta + a\|y\| \|F(V_{\delta,a})-f_\delta\|\\
&\le \delta^2 + a\|y\|\delta + \epsilon \|F(V_{\delta,a})-f_\delta\|^2 + 
\frac{1}{4\epsilon}a^2\|y\|^2,
\end{split}
\end{equation}
where $\epsilon>0$ is arbitrary small, fixed, independent of $a$, and can be chosen 
arbitrary small. 
Let $a\searrow 0$. Then \eqref{3eq4} implies
$\lim_{a\to 0}(1-\epsilon)\|F(V_{\delta,a})-f_\delta\|^2\le \delta^2<(C\delta^\gamma)^2$. 
Thus,
$$
\lim_{a\to 0} \|F(V_{\delta,a})-f_\delta\| < C\delta^\gamma,\qquad C>0,\quad 0<\gamma\le 1.
$$
This, the continuity of $F$, the continuity of $V_{\delta,a}$ with respect to $a \in [0,\infty)$, 
and inequality \eqref{eqrem1},  
imply that equation $\|F(V_{\delta,a})-f_\delta\|=C\delta^\gamma$ must have a 
solution $a(\delta) >0$.
\end{proof}

\begin{rem}{\rm 
Let $V_a:=V_{\delta,a}|_{\delta=0}$, so $F(V_a)+aV-f=0$. 
Let $y$ be the minimal-norm solution to equation \eqref{eq1}. 
We claim that
\begin{equation}
\label{eq13}
\|V_{\delta,a}-V_a\|\le \frac{\delta}{a}.
\end{equation}
Indeed, from \eqref{eq3} one gets
$$
F(V_{\delta,a}) - F(V_a) + a (V_{\delta,a}-V_a)=f- f_\delta.
$$
Multiply this equality by $(V_{\delta,a}-V_a)$ and use \eqref{eq2} to obtain
\begin{align*}
\delta \|V_{\delta,a}-V_a\| &\ge \langle f-f_\delta, V_{\delta,a}-V_a \rangle\\
&= \langle F(V_{\delta,a}) - F(V_a) + a (V_{\delta,a}-V_a), V_{\delta,a}-V_a) \rangle\\
&\ge a \|V_{\delta,a}-V_a\|^2.
\end{align*}
This implies \eqref{eq13}. 
}
\end{rem}

Let us derive a uniform with respect to $a$ bound on $\|V_a\|$. From the equation 
$$
F(V_a) + a V_a -F(y)=0,
$$
and the monotonicity of $F$ one gets
$$
0=\langle F(V_a) + a V_a -F(y), V_a - y \rangle \ge a \langle V_a, V_a-y\rangle.
$$
This implies the desired bound:
\begin{equation}
\label{eq14.1}
\|V_a\| \le \|y\|,\qquad \forall a>0.
\end{equation}
Similar arguments one can find in \cite[p. 113]{R499}. 

From \eqref{eq13} and \eqref{eq14.1}, one gets the following estimate:
\begin{equation}
\label{2eq1}
\|V_{\delta,a}\|\le \|V_a\|+\frac{\delta}{a}\le \|y\|+\frac{\delta}{a}.
\end{equation}

\section{A discrepancy principle}
\label{sec3}

Our standing assumptions are the monotonicity and continuity of $F$
and the solvability of equation \eqref{eq1}. 
They are not repeated below. We assume without loss of generality that $\delta\in (0,1)$.
\begin{thm}
\label{thm1}
Let $\gamma\in (0,1]$ and $C>0$ be some constants such that $C\delta^\gamma>\delta$.
Assume that 
$\|F(0)-f_\delta\|>C\delta^\gamma$. Let $y$ be its minimal-norm solution.
Then there exists a unique $a(\delta)>0$ such that
\begin{equation}
\label{eq131}
\|F(V_{\delta,a(\delta)}) - f_\delta \| = C\delta^\gamma,
\end{equation}
where $V_{\delta, a(\delta)}$ solves \eqref{eq3} with $a = a(\delta)$.

If $0< \gamma <1$ then
\begin{equation}
\label{eq14}
\lim_{\delta\to 0} \|V_{\delta, a(\delta)}-y\| = 0.
\end{equation}
\end{thm}

\begin{proof}
The existence and uniqueness of $a(\delta)$ follow from Lemma~\ref{lem4}. 
Let us show that
\begin{equation}
\label{eqx16}
\lim_{\delta\to0} a(\delta) = 0.
\end{equation}
The triangle inequality, inequality \eqref{eq13} and equality \eqref{eq131} imply
\begin{equation}
\label{hutchet}
\begin{split}
a(\delta)\|V_{a(\delta)}\| &\le a(\delta) \big{(}\|V_{\delta,a(\delta)} - V_{a(\delta)}\| + \|V_{\delta,a(\delta)}\|\big{)}\\
&\le \delta + a(\delta)\|V_{\delta,a(\delta)}\| = \delta + C \delta^\gamma.
\end{split}
\end{equation}
From inequality \eqref{hutchet}, one gets
\begin{equation}
\label{xxtichet}
\lim_{\delta\to 0}a(\delta)\|V_{a(\delta)}\| = 0.
\end{equation}
It follows from Lemma~\ref{lem2} with $f_\delta= f$, i.e., $\delta=0$, that 
 the function $\phi_0(a):=a\|V_{a}\|$ is nonnegative and strictly increasing on $(0,\infty)$.
 This and relation \eqref{xxtichet} imply: 
 \begin{equation}
\label{xxeqx28}
\lim_{\delta\to 0} a(\delta)= 0.
\end{equation}

From \eqref{eq131} and \eqref{2eq1}, one gets
\begin{equation}
C\delta^\gamma = a\|V_{\delta,a}\|\le a(\delta)\|y\| + \delta.
\end{equation}
Thus, 
%since $\gamma\le 1$ and $\delta<C\delta^\gamma$ for sufficiently small $\delta$, 
one gets:
\begin{equation}
 C\delta^\gamma - \delta \le a(\delta)\|y\|.
\end{equation}
If $\gamma<1$ then $C-\delta^{1-\gamma}>0$ for sufficiently small $\delta$.  This implies: 
\begin{equation}
\label{eq15}
0\le \lim_{\delta\to 0}\frac{\delta}{a(\delta)}\le \lim_{\delta\to 0}\frac{\delta^{1-\gamma}\|y\|}{C-\delta^{1-\gamma}} = 0. 
\end{equation}
By the triangle inequality and inequality \eqref{eq13}, one has
\begin{equation}
\label{eq16}
\|V_{\delta, a(\delta)} - y\|\le \|V_{a(\delta)} - y\| + \|V_{a(\delta)}-V_{\delta,a(\delta)}\| 
\le \|V_{a(\delta)} - y\| + \frac{\delta}{a(\delta)}.
\end{equation}
Relation \eqref{eq14} follows from \eqref{eq15}, \eqref{eq16} and 
Lemma~\ref{lem1}.
\end{proof}

Instead of using \eqref{eq3}, one may use the following equation:
\begin{equation}
\label{}
F(V_{\delta,a})+a(V_{\delta,a}-\bar{u})-f_\delta = 0,\qquad a>0,
\end{equation}
where $\bar{u}$ is an element of $H$. Denote $F_1(u):=F(u+\bar{u})$.
 Then $F_1$ is monotone and continuous. 
Equation \eqref{eq3} can be written as:
\begin{equation}
\label{}
F_1(U_{\delta,a})+aU_{\delta,a}-f_\delta = 0,
\qquad U_{\delta,a}:=V_{\delta,a}-\bar{u},\quad a>0.
\end{equation}
By applying Theorem~\ref{thm1} with $F=F_1$ one gets the following result:

\begin{cor}
Let $\gamma\in (0,1]$ and $C>0$ be some constants such that $C\delta^\gamma>\delta$.
Let $\bar{u}\in H$ and $z$ be the solution to \eqref{eq1} with 
minimal distance to $\bar{u}$.
Assume that 
$\|F(\bar{u})-f_\delta\|>C\delta^\gamma$. 
Then there exists a unique $a(\delta)>0$ such that
\begin{equation}
\|F(\tilde{V}_{\delta,a(\delta)}) - f_\delta \| = C\delta^\gamma,
\end{equation}
where $\tilde{V}_{\delta, a(\delta)}$ solves the following equation:
\begin{equation*}
%\label{}
F(\tilde{V}_{\delta,a})+a(\delta)(\tilde{V}_{\delta,a}-\bar{u})-f_\delta = 0.
\end{equation*}

If $\gamma \in (0,1)$ then this $a(\delta)$ satisfies
\begin{equation}
\lim_{\delta\to 0} \|\tilde{V}_{\delta, a(\delta)}-z\| = 0.
\end{equation}
\end{cor}

\begin{rem}{\rm
It is an open problem to choose $\gamma$ and $C$ optimal in some 
sense.
}
\end{rem}

\begin{rem}{\rm 
Theorem~\ref{thm1} and Theorem~\ref{thm2} do not hold, in general, for $\gamma=1$.
Indeed, let $Fu = \langle u,p\rangle p$, 
$\|p=1\|,\, p\perp \mathcal{N}(F):=\{u\in H: Fu=0\}$, $f=p$, $f_\delta=p+q\delta$,
where $\langle p,q\rangle=0$, $\|q\|=1$, $Fq=0$, $\|q\delta\|=\delta$.
One has $Fy=p$, where $y=p$, is the minimal-norm solution to the equation $Fu=p$.
Equation $Fu+au=p+q\delta$, has the unique solution $V_{\delta,a}=q\delta/a + p/(1+a)$.
Equation \eqref{eq131} is $C\delta=\|q\delta + (ap)/(1+a)\|$. This equation yields $a=a(\delta)=c\delta/(1-c\delta)$, where $c:=(C^2-1)^{1/2}$, and we assume $c\delta<1$.
Thus, $\lim_{\delta\to0}V_{\delta,a(\delta)}=p+c^{-1}q:=v$, and $Fv=p$.
Therefore $v=\lim_{\delta\to0}V_{\delta,a(\delta)}$ is not $p$, i.e., is not the minimal-norm 
solution to the equation $Fu=p$. 
Similar arguments one can find in \cite[p. 29]{R470}.
}
\end{rem}

\section{Applications}
\label{sec4}

In this section we discuss methods for solving equations \eqref{eq3} and 
\eqref{eq1} using the new discrepancy principle, i.e., Theorem~\ref{thm1}. 
Implementing this principle, i.e., solving equation \eqref{eq131}, 
requires solving equation \eqref{eq3}. 
{\it If $F$
 is linear}, then equation \eqref{eq3} has the form:
\begin{equation}
\label{eq21}
(F+aI)u = f_\delta.
\end{equation}
Since $F\ge 0$ the operator $F+aI$ is boundedly invertible, 
$\|(F+aI)^{-1}\| \le \frac{1}{a}$, and  
equation \eqref{eq21} is well-posed if $a>0$ is not too small. 
There are many methods for solving efficiently well-posed linear equations
with positive-definite operators.
For this reason we {\it mainly discuss some methods for stable solution of equation \eqref{eq1} with
 nonlinear operators.} In this section a method is developed 
for a stable solution of equation \eqref{eq1} with locally Lipschitz monotone operator $F$,
so we assume that
\begin{equation}
\label{eqz23}
\|F(u)-F(v)||\leq L||u-v||,\quad u,v\in B(u_0,R):=\{u:\|u-u_0\|\le R\},
\quad L=L(R).
\end{equation} 
Here $u_0\in H$ is an arbitrary fixed element.
Consider the operator
$$
G(u):= u - \lambda [F(u) + a u -f_\delta],\quad \lambda >0 .
$$
We claim that $G$ is a contraction mapping in $H$ provided that $\lambda$ 
is sufficiently small.
Let $F_1:=F+aI$. Then \eqref{eqz23} implies $\|F_1(u)-F_1(v)\|\le 
(a+L)\|u-v\|$.
Using the monotonicity of $F$, one gets
\begin{equation}
\begin{split}
\|G(u)-G(v)\|^2 &= \|(u-v) - \lambda \big{(}F_1(u)- F_1(v)\big{)}\|^2\\
&= \|u-v\|^2 - 2\lambda \langle u-v, F_1(u)- F_1(v)\rangle + \lambda^2\|F_1(u)- F_1(v)\|^2\\
&\le \|u-v\|^2[1 - 2\lambda a + \lambda^2(a+L)^2].
\end{split}
\end{equation}
This implies that $G$ is a contraction mapping if
$$
0 < \lambda < \frac{2a}{(a+L)^2}.
$$
For these $\lambda$ the solution $V_{\delta,a}$ of equation \eqref{eq3} can be 
found by the following iterative process:
\begin{equation}
\label{eq22}
u_{n+1} = u_n - \lambda [F(u_n)+a u_n -f_\delta],\quad u_0:= u_0.
\end{equation} 
After finding $V_{\delta,a}$, one finds $a(\delta)$ from 
the discrepancy principle \eqref{eq131},
i.e., by solving the nonlinear equation:
\begin{equation}
\label{eq23}
\phi(a(\delta)):=\|F(V_{\delta,a(\delta)}) - f_\delta\| = C\delta^\gamma.
\end{equation}
There are many methods for solving this equation. For example, one can use 
the bisection method or the golden section method. 
If $a(\delta)$ is found, one solves equation \eqref{eq3} with 
$a=a(\delta)$ for $V_{\delta,a(\delta)}$
and takes its solution as an approximate solution to \eqref{eq1}.

Although the sequence $u_n$, defined by \eqref{eq22}, converges to the 
solution of equation \eqref{eq3} at the rate of
a geometrical series with a denominator $q\in(0,1)$, 
it is very time consuming to try to solve 
equation \eqref{eq3} with high accuracy if $q$ is close to 1. 
Theorem~\ref{thm2} (see below) allows one to 
%solve equation \eqref{eq3} approximately. Indeed, we can 
stop iterations \eqref{eq22} at the first value of $n$ which satisfies the 
following condition:
\begin{equation}
\label{eq230}
\|F(u_n) + a u_n - f_\delta\| \le \theta \delta,\qquad \theta>0,
\end{equation}
where 
%%$C>1$ is a constant from \eqref{eq23} and 
$\theta$ is a fixed constant.
This saves the time of computation.

\begin{thm}
\label{thm2}
Let $\delta, F, f_\delta$, and $y$ be as in Theorem~\ref{thm1} and $0<\gamma<1$. 
Assume that $v_\delta\in H$ and $\alpha(\delta)>0$ 
satisfy the following conditions:
\begin{equation}
\label{eqx23}
\|F(v_\delta)+\alpha(\delta) v_\delta - f_\delta\| \le \theta \delta,\qquad \theta>0,
%\quad 0<\theta<\frac{C_1-1}{2}, 
\end{equation}
and
\begin{equation}
\label{eqx24}
C_1\delta^\gamma \le \|F(v_\delta) - f_\delta\| \le C_2 \delta^\gamma,
%\quad 0<\gamma<1,
%\quad 1+2\theta<C_1\le C-\theta< 
\qquad
0< C_1 < C_2.
%C+\theta \le C_2. 
\end{equation}
%%%%%If conditions  \eqref{eqx23}--\eqref{eqx25} hold, then
Then one has:
\begin{equation}
\label{eqx25}
\lim_{\delta\to 0}\|v_\delta - y\| = 0.
\end{equation}
\end{thm}

\begin{proof}
Let $u$ and $v$ be arbitrary elements in $H$. By the monotonicity of 
$F$ one gets
\begin{equation}
\begin{split}
a\|u-v\|^2 &\le \langle u-v, F(u)-F(v)+ au-av \rangle\\
&\le \|u-v\| \|F(u)-F(v)+ au-av\|,\qquad \forall a> 0.
\end{split}
\end{equation}
This implies
\begin{equation}
\label{eqy28}
a\|u-v\| \le \|F(u)-F(v)+ au-av\|,\qquad \forall v,u \in H,\quad \forall a>0.
\end{equation}
Using inequality \eqref{eqy28} with $v=v_\delta$ and $u=V_{\delta,\alpha(\delta)}$, 
equation \eqref{eq3} with $a=\alpha(\delta)$, and inequality \eqref{eqx23},
one gets
\begin{equation}
\begin{split}
\alpha(\delta) \|v_\delta - V_{\delta,\alpha(\delta)}\| &\le \|F(v_\delta)-F(V_{\delta,\alpha(\delta)})+\alpha(\delta) v_\delta - \alpha(\delta) V_{\delta,\alpha(\delta)}\|\\
&=\|F(v_\delta)+ \alpha(\delta) v_\delta -f_\delta\| \le \theta\delta.
\end{split}
\end{equation}
Therefore,
\begin{equation}
\label{eqx26}
\|v_\delta - V_{\delta,\alpha(\delta)}\| \le \frac{\theta \delta}{\alpha(\delta)}.
\end{equation}
Using \eqref{2eq1} and \eqref{eqx26}, one gets:
\begin{equation}
\label{eqy30}
\alpha(\delta) \|v_\delta\| \le \alpha(\delta) \|V_{\delta,\alpha(\delta)}\| + 
\alpha(\delta)\|v_\delta - V_{\delta,\alpha(\delta)}\| \le \theta \delta + 
\alpha(\delta) \|y\| + \delta.
\end{equation}
From the triangle inequality and inequalities \eqref{eqx23} and  
\eqref{eqx24} one obtains: 
\begin{equation}
\label{eqy31}
\alpha(\delta)\|v_\delta\| \ge \|F(v_\delta) - f_\delta\| - \|F(v_\delta)+ \alpha(\delta) v_\delta -f_\delta\|
\ge C_1\delta^\gamma - \theta \delta.
\end{equation}
Inequalities \eqref{eqy30} and \eqref{eqy31} imply
\begin{equation}
C_1\delta^\gamma - \theta \delta \le \theta \delta + 
\alpha(\delta) \|y\| + \delta.
\end{equation}
This inequality and the fact that $C_1-\delta^{1-\gamma} - 2\theta\delta^{1-\gamma}>0$ for sufficiently small $\delta$ and $0<\gamma<1$ imply
\begin{equation}
%\label{}
\frac{\delta}{\alpha(\delta)} \le \frac{\delta^{1-\gamma}\|y\|}
{C_1-\delta^{1-\gamma} - 2\theta\delta^{1-\gamma}}, \qquad  0<\delta\ll 1.
\end{equation}
Thus, one obtains 
\begin{equation}
\label{eqx27}
\lim_{\delta\to 0} \frac{\delta}{\alpha(\delta)} = 0.
\end{equation}
From the triangle inequality and inequalities \eqref{eqx23}, \eqref{eqx24} and \eqref{eqx26}, one gets
\begin{equation}
%\label{}
\begin{split}
\alpha(\delta)\|V_{\delta,\alpha(\delta)}\|
&\le \|F(v_\delta) - f_\delta\| + \|F(v_\delta)+\alpha(\delta)v_\delta -f_\delta\| + 
\alpha(\delta)\|v_\delta - V_{\delta,\alpha(\delta)}\|\\
&\le C_2 \delta^\gamma + \theta \delta + \theta \delta.
\end{split}
\end{equation}
%%%%%%%%%%
This inequality implies
\begin{equation}
\label{eqxx34}
\lim_{\delta\to 0}\alpha(\delta)\|V_{\delta,\alpha(\delta)}\| = 0.
\end{equation}
The triangle inequality and inequality \eqref{eq13} imply
\begin{equation}
\label{chet}
\begin{split}
\alpha\|V_{\alpha}\| &\le \alpha \big{(}\|V_{\delta,\alpha} - V_{\alpha}\| + \|V_{\delta,\alpha}\|\big{)}\\
&\le \delta + \alpha\|V_{\delta,\alpha}\|.
\end{split}
\end{equation}
From formulas \eqref{chet} and \eqref{eqxx34}, one gets
\begin{equation}
\label{tichet}
\lim_{\delta\to 0}\alpha(\delta)\|V_{\alpha(\delta)}\| = 0.
\end{equation}
It follows from Lemma~\ref{lem2} with $f_\delta= f$, i.e., $\delta=0$, that 
 the function $\phi_0(a):=a\|V_{a}\|$ is nonnegative and strictly increasing on $(0,\infty)$.
 This and relation \eqref{tichet} imply 
 \begin{equation}
\label{eqx28}
\lim_{\delta\to 0} \alpha(\delta)= 0.
\end{equation}
%%%%
From the triangle inequality and inequalities \eqref{eqx26} 
and \eqref{eq13} one obtains
\begin{equation}
\label{eqx29}
\begin{split}
\|v_\delta - y\| &\le \|v_\delta - V_{\delta,\alpha(\delta)}\| + 
\|V_{\delta,\alpha(\delta)} - V_{\alpha(\delta)}\| + \|V_{\alpha(\delta)} 
- y\|\\
&\le \frac{\theta \delta}{\alpha(\delta)} + \frac{\delta}{\alpha(\delta)} +  
\|V_{\alpha(\delta)} - y\|,
\end{split}
\end{equation}
where $V_{\alpha(\delta)}$ solves equation (3) with $a=\alpha(\delta)$
and $f_\delta=f$.

The conclusion \eqref{eqx25} follows from inequalities 
\eqref{eqx27}, \eqref{eqx28}, \eqref{eqx29} and Lemma~\ref{lem1}.
Theorem~\ref{thm2} is proved. 
\end{proof}

\begin{rem}{\rm 
Inequalities \eqref{eqx23} and \eqref{eqx24} are used as stopping rules for 
finding approximations:  
\begin{equation*}
\alpha(\delta)\approx a(\delta),\quad \text{and} \quad v(\delta)\approx V_{\delta,a(\delta)}.
\end{equation*}
}
\end{rem}

\begin{rem}{\rm
By the 
monotonicity of $F$ one gets
\begin{equation*}
\begin{split}
\|F(u)-F(v)\|^2 &\le \langle F(u)-F(v), F(u)-F(v)+ a(u-v)\rangle \\
&\le \|F(u)-F(v)\|\|F(u)-F(v)+ a(u-v)\|,\quad \forall u,v\in H,\quad \forall a>0.
\end{split}
\end{equation*}
This implies
\begin{equation}
\label{eqxx36}
\|F(u)-F(v)\| \le \|F(u)-F(v)+ a(u-v)\|,\qquad \forall u,v\in H,\quad a>0.
\end{equation}
}
\end{rem}

Fix $\delta>0$ and $\theta>0$. Let $C$ be as in Theorem~\ref{thm1}. Choose $C_1$ and $C_2$ such that 
\begin{equation}
\label{neq54}
C_1\delta^\gamma + \theta\delta< C\delta^\gamma< C_2\delta^\gamma - \theta\delta.
\end{equation}
Suppose $\alpha_i$ and $v_i$, $i=1,2,$ satisfy condition 
\eqref{eqx23} and 
\begin{equation}
\label{neq55}
\|F(v_1) - f_\delta\| < C_1\delta^\gamma,\qquad C_2\delta^\gamma< \|F(v_2) - f_\delta\|.
\end{equation}
Let us show that
\begin{equation}
\label{neq56}
\alpha_{low}:=\alpha_1 < a(\delta) < \alpha_2:=\alpha_{up},
\end{equation} 
where $a(\delta)$ satisfies conditions of Theorem~\ref{thm1}.
Using inequality \eqref{eqxx36} for $v_i$ and $V_{\delta,\alpha_i}$, $i=1,2$, 
and inequality \eqref{eqx23}, one gets
\begin{equation}
\label{neq57}
\begin{split}
\|F(v_i)-F(V_{\delta,\alpha_i})\| &\le \|F(v_i)-F(V_{\delta,\alpha_i}) + \alpha_i v_i
- \alpha_i V_{\delta,\alpha_i}\|\\
&\le \|F(v_i) + \alpha_i v_i
- f_\delta\| \le \theta\delta.
\end{split}
\end{equation}
From inequalities \eqref{neq55}, \eqref{neq57} and the triangle inequality, one derives:
\begin{equation}
\label{neq58}
\|F(V_{\delta,\alpha_1})-f_\delta\| < C_1\delta^\gamma + \theta\delta\quad \text{and}\quad 
 C_2\delta^\gamma - \theta\delta <\|F(V_{\delta,\alpha_2})-f_\delta\|.
\end{equation}
Recall that $\|F(V_{\delta,a(\delta)})-f_\delta\| = C\delta^\gamma$. 
Inequality \eqref{neq56} is obtained from inequalities \eqref{neq54}, \eqref{neq58} and the fact that the function $\phi(\alpha)=\|F(V_{\delta,\alpha})-f_\delta\|$ is strictly increasing (see Lemma~\ref{lem2}).

Let $f_\delta,F,C,\theta,\gamma$, and $\delta$ be as in Theorem~\ref{thm1} and \ref{thm2},
and $C_1$ and $C_2$ satisfy inequality \eqref{neq54}. 
Let us formulate an algorithm (see {\bf Algorithm 1} below) 
for finding $\alpha(\delta)\approx a(\delta)$ and 
$v(\delta)\approx V_{\delta,a(\delta)}$, using the bisection method and
assuming that $F$ is a locally Lipschitz monotone operator 
and $\alpha_{low}$ and $\alpha_{up}$ are known. 
By Theorem~\ref{thm2}, $v(\delta)$ can be considered as a stable solution 
to equation \eqref{eq1}.

{\bf Algorithm 1}: 
{\it Finding $\alpha(\delta)\approx a(\delta)$ and $v_\delta\approx V_{\delta,a(\delta)}$ given $\alpha_{low}$ and $\alpha_{up}$.}
\begin{enumerate}

\item 
Let $a:=\frac{\alpha_{up}+\alpha_{low}}{2}$ and 
$u_0$ be an initial guess for $V_{\delta,a}$.
Compute $u_n$ by formula \eqref{eq22} and stop at 
$n_{stop}$, where $n_{stop}$ is the smallest $n>0$
 for which condition \eqref{eqx23} is satisfied. Then go to step 2.
 
\item
If $C_2\delta^\gamma < \|F(u_{n_{stop}})-f_\delta\|$,
then set  $\alpha_{up}:=a$ and go to step 4.
Otherwise, go to step 3.

\item
If $C_1\delta^\gamma\le\|F(u_{n_{stop}})-f_\delta\|$, then stop the process 
and take $v(\delta):=u_{n_{stop}}$ as a solution to \eqref{eq1}.
If $\|F(u_{n_{stop}})-f_\delta\|<C_1\delta^\gamma$,
then set  $\alpha_{low}:=a$ and go to step 4.

\item Check if $\|a-\alpha_{low}\|$ is less than a desirable small 
value $\epsilon>0$. If it is,  then
take $v(\delta):=u_{n_{stop}}$ as a solution to \eqref{eq1}. 
If is is not, then go back to step 1.
\end{enumerate}

Let us formulate algorithms for finding $\alpha_{up}$ and $\alpha_{low}$.

{\bf Algorithm 2}: {\it Finding $\alpha_{up}$.}
\begin{enumerate}
\item Let $a=\alpha$ be an initial guess for $\alpha(\delta)$ and $u_0$ be an initial guess for $v_\delta$. 
Compute $u_n$ by formula \eqref{eq22} with $a$ and stop at 
$n_{stop}$, the smallest $n>0$
 for which condition \eqref{eqx23} is satisfied. Then go to step 2.
 
\item If condition \eqref{eqx24} holds for $v_\delta:= u_{n_{stop}}$, then stop 
the process and take $u_{n_{stop}}$ as a solution to \eqref{eq1}.
Otherwise, go to step 3.

\item 
If $C_2\delta^\gamma < \|F(u_{n_{stop}})-f_\delta\|$,
then set  $\alpha_{up}:=a$. Otherwise, set $\alpha:=2a$ and go back step 1.
\end{enumerate}

{\bf Algorithm 3}: {\it Finding $\alpha_{low}$.}
\begin{enumerate}
\item Let $a=\alpha$ be an initial guess for $\alpha(\delta)$ and $u_0$ be an initial guess for $v_\delta$. 
Compute $u_n$ by formula \eqref{eq22} with $a$ and stop at 
$n_{stop}$, the smallest $n>0$
 for which condition \eqref{eqx23} is satisfied. Then go to step 2.
 
\item If condition \eqref{eqx24} holds for $v_\delta:= u_{n_{stop}}$, then stop 
the process and take $u_{n_{stop}}$ as a solution to \eqref{eq1}.
Otherwise, go to step 3.

\item 
If $\|F(u_{n_{stop}})-f_\delta\| < C_1\delta^\gamma$,
then set  $\alpha_{low}:=a$. Otherwise, set $\alpha:=\frac{a}{2}$ and go back step 1.
\end{enumerate}

In practice these algorithms are often implemented at the same time to avoid repetition calculations.

\begin{rem}{\rm
The sequence $(\|u_n - V_{\delta,a(\delta)}\|)_{n=0}^\infty$, where
$u_n$ is computed by formula \eqref{eq22} and $V_{\delta,a(\delta)}$
is the solution to \eqref{eq3} with $a=a(\delta)$, is decreasing.
Thus, the sequence $u_n$ will stay
inside a ball $B(0,R)$ assuming that $R>0$ is chosen sufficiently large, so that $y,u_0\in B(0,R)$.
}
\end{rem}

\begin{rem}{\rm
Theorem~\ref{thm2} and the above algorithms are not only useful for solving 
nonlinear equations with monotone operators
but also for solving linear equations with monotone operators. 
If one uses iterative methods to solve equation \eqref{eq21} then, by using
Theorem~\ref{thm2}, one can stop iterations whenever inequality \eqref{eqx23} holds.
 By using stopping rule \eqref{eqx23} one saves time of computations compared to solving \eqref{eq21} exactly. 
If $F$ is a positive matrix then one can solve \eqref{eq21} by 
conjugate gradient, or Jacobi, or Gauss-Seidel, or successive over-relaxation 
methods, with stopping rule \eqref{eqx23}.
}
\end{rem}

%\subsection{Equations with Fr\'{e}chet differentiable operators}
\begin{rem}{\rm
If $F$ is twice Fr\'{e}chet 
differentiable, 
there are more options for solving 
equations \eqref{eq3} and \eqref{eq23}:
they can be solved by 
gradient-type methods, Newton-type methods, or a combination of these 
methods.
}
\end{rem}


\begin{thebibliography}{99}

\bibitem{D}
K. Deimling, Nonlinear functional analysis, Springer Verlag, Berlin, 1985.

%\bibitem{R550} N. S. Hoang and A. G. Ramm, Dynamical system method for
%solving nonlinear equations with monotone operators, Diff. Eq. Appl. (to
%appear).

\bibitem{Lions} J. L. Lions, Quelques methodes de resolution des problemes
aux limites non lineaires, Dunod, Gauthier-Villars, Paris, 1969.

\bibitem{Pascali} D. Pascali and S. Sburlan, Nonlinear Mappings of
Monotone Type, Noordhoff, Leyden, 1978.

\bibitem{[8]} A. G. Ramm, Theory and applications of some new classes of
integral equations, Springer-Verlag, New York, 1980.

\bibitem{[9]} A. G. Ramm, Stationary regimes in passive nonlinear
networks, in the book "Nonlinear Electromagnetics", Ed. P.Uslenghi, Acad.
Press, New York, 1980, pp. 263-302.

\bibitem{[19]} A. G. Ramm, Iterative solution of linear equations with
unbounded operators, J. Math. Anal. Appl., 1338-1346.

\bibitem{[20]} A. G. Ramm, On unbounded operators and applications, Appl.
Math. Lett., 21, (2008), 377-382.

\bibitem{R470} A. G. Ramm, Inverse problems, Springer, New York, 2005.

\bibitem{R499} A. G. Ramm, Dynamical systems method for solving
operator equations, Elsevier, Amsterdam, 2007.

\bibitem{Skrypnik} I. V. Skrypnik, Methods for Analysis of Nonlinear
Elliptic Boundary Value Problems, American Mathematical Society,
Providence, RI, 1994.


\bibitem{Tautenhahn} U. Tautenhahn, On the method of Lavrentiev
regularization for nonlinear ill-posed problems, Inverse Probl., 18,
(2002), 191-207.


\bibitem{Vainberg} M. M. Vainberg, Variational methods and method of
monotone operators in the theory of nonlinear equations, Wiley, London,
1973.


%\bibitem{Z} E. Zeidler, Nonlinear functional analysis, Springer, New
%York, 1985.


\end{thebibliography}
\end{document}